\documentclass[12pt,oneside,reqno]{amsart}
\usepackage{amsfonts,amsmath,amscd,amssymb,paralist}
\usepackage{mathrsfs}
\usepackage{amsxtra}
\usepackage[mathscr]{eucal}
\usepackage{graphics}
\usepackage{latexsym,graphicx,verbatim,nameref}
\usepackage[numbers,sort&compress]{natbib}
\usepackage[all,cmtip]{xy}
\usepackage{hyperref}
\usepackage{fancyhdr}
\usepackage{color}
\usepackage{tikz}
\usepackage{tikz-cd}
\usepackage{graphicx}
\usetikzlibrary{arrows,scopes,matrix}
\numberwithin{equation}{section}

\makeatletter
\newtheorem*{rep@theorem}{\rep@title}
\newcommand{\newreptheorem}[2]{%
\newenvironment{rep#1}[1]{%
 \def\rep@title{#2 \ref{##1}}%
 \begin{rep@theorem}}%
 {\end{rep@theorem}}}
\makeatother

\newreptheorem{theorem}{Theorem}
\newreptheorem{corollary}{Corollary}
\newreptheorem{lemma}{Lemma}
\newreptheorem{conjecture}{Conjecture}

\theoremstyle{definition}
\newtheorem{theorem}{Theorem}[section]
\newtheorem{corollary}[theorem]{Corollary}
\newtheorem{lemma}[theorem]{Lemma}
\newtheorem{proposition}[theorem]{Proposition}
\newtheorem*{theorem*}{Theorem}
\newtheorem*{proposition*}{Proposition}
\newtheorem{definition}[theorem]{Definition}
\newtheorem{remark}[theorem]{Remark}

\newtheorem*{claim*}{Claim}
\newtheorem{conjecture}[theorem]{Conjecture}

\newtheorem*{conjecture*}{Conjecture}
\newtheorem*{observation*}{Observation}
\newtheorem*{question*}{Question}

\begin{document}
\title{Continuous $\times p,\times q$-invariant measures on the unit circle}
\author{Huichi Huang}
\address{College of Mathematics and Statistics, Chongqing University, Chongqing, 401331, PR. China}
\email{huanghuichi@cqu.edu.cn}
\date{\today}
\begin{abstract}
We  express continuous $\times p,\times q$-invariant measures on the unit circle via some simple forms. On one hand, a continuous $\times p,\times q$-invariant measure is the weak-$*$ limit of average of Dirac measures along an irrational orbit. On the other hand, a continuous $\times p,\times q$-invariant measure is a continuous function on $[0,1]$ satisfying certain function equations.
\end{abstract}

\maketitle
\tableofcontents

\section{Introduction}\

In ~\cite{Furstenberg1967}, H. Furstenberg shows that  when $\frac{\log{p}}{\log{q}}$ is irrational, every irrational orbit under $\times p,\times q$ is dense in the unit circle $\mathbb{T}$. He also conjectures that the only continuous ergodic $\times p,\times q$-invariant measure is the  Lebesgue measure.

In this paper, we express continuous $\times p,\times q$-invariant measures on the unit circle via two simple forms. One is an average of Dirac measures and the other one is homeomorphisms on $[0,1]$.

The first says the following.
\begin{reptheorem}{IrrationalOrbit}
If $\mu$ is an ergodic $\times p,\times q$-invariant continuous Borel probability measure on $\mathbb{T}$, then there exists an irrational $x\in[0,1)$ such that $$\lim_{N\to\infty}\dfrac{1}{N^2}\sum_{i=0}^{N-1}\sum_{j=0}^{N-1}\delta_{p^iq^jx}=\mu$$ under weak-$*$ topology, where $\delta_y$ is the Dirac measure on $[0,1)$ concentrating at a point $y\in[0,1)$.
\end{reptheorem}

The second is a conjecture equivalent to Furstenberg's conjecture.

\begin{repconjecture}{Cont}
The only homeomorphism $f$ on $[0,1]$ satisfying $f=T_p f=T_q f$ is the identity. Here for a positive integer $n$, the operator $T_n: C[0,1]\to C[0,1]$ is given by
$T_n g(x)=\sum_{i=0}^{n-1} g(\frac{x+i}{n})-g(\frac{i}{n})$ for every $x\in[0,1]$ and $g\in C[0,1]$.
\end{repconjecture}
This can be taken as a real-value function version of~\cite[Prop. 11]{Deninger2011}.
\section{Preliminary}\

\subsection{Conventions}\
Within this article, we denote the unit circle $\{z\in\mathbb{C}|\,|z|=1\}$ by $\mathbb{T}$~(if necessary $\mathbb{T}$ will be also denoted by $\mathbb{R}/\mathbb{Z}$). Denote the set of nonnegative integers by $\mathbb{N}$, the set of positive integers by $\mathbb{Z}^+$ and the function $\exp{2\pi ix}$ for $x\in\mathbb{R}$ by $e(x)$ and the function $e(kx)$ by $z^k$ for every $k\in\mathbb{Z}$. The notation $C(X)$ stands for  the set of continuous functions on a compact Hausdorff space $X$.

A measure always means a Borel probability measure. By identifying $\mathbb{T}$ with $[0,1)$, a measure on $\mathbb{T}$ amounts to a measure on $[0,1)$.

We call a number $a\in\mathbb{T}$ rational if $a=e(x)$ for some rational $x\in [0,1)$, otherwise call $a$ irrational. The greatest common divisor of $m,n\in\mathbb{Z}^+$ is denoted by $\gcd(m,n)$.

Let $\omega=\{x_n\}_{n=1}^\infty$ be a sequence of real numbers contained in the unit interval $[0,1)$ and for any positive integer $N$ and a subset $E\subseteq [0,1)$, denote
$\frac{|\{x_1,\cdots,x_N\}\cap E|}{N}$ by $A(E; N;\omega)$ or briefly $A(E; N)$ if no confusion caused.

For a double sequence $\omega=\{s_{ij}\}_{i,j=0}^\infty\subseteq [0,1)$, positive integers $N,M$ and a subset $E\subseteq [0,1)$, denote $\frac{|\{s_{ij}|0\leq i\leq M-1,0\leq j\leq N-1\}\cap E|}{NM}$ by $A(E; N,M;\omega)$  or briefly $A(E; N,M)$.

\subsection{Equidistributed sequences in $\mathbb{T}$}\
\begin{definition}~[Equidistributed (double) sequences]\

A sequence $\{a_n\}_{n=1}^\infty$ in $\mathbb{T}$ is called {\bf equidistributed}  if the sequence $\omega=\{x_n\}_{n=1}^\infty$ in $[0,1)$ with $e(x_n)=a_n$ satisfies
$$\lim_{N\to\infty}A([a,b); N;\omega)=b-a,$$ for any $0\leq a<b\leq 1$, or equivalently one can say  the sequence $\{x_n\}_{n=1}^\infty$ is uniformly distributed modulo 1~( u.d. $\mod{1}$)~\cite[Defn. 1.1]{KuipersNiederreiter1974}.

A double sequence $\{a_{i,j}\}_{i,j=0}^\infty$ in $\mathbb{T}$ is called {\bf equidistributed}if the sequence $\omega=\{x_{ij}\}_{i,j=0}^\infty$ in $[0,1)$ such that $e(x_{ij})=a_{ij}$ satisfies
$$\lim_{N,M\to\infty}A([a,b); N,M;\omega)=b-a,$$ for any $0\leq a<b\leq 1$, or equivalently one can say  the sequence $\{x_{i,j}\}_{i,j=0}^\infty$ is uniformly distributed modulo 1~( u.d. $\mod{1}$)~\cite[Defn. 2.1]{KuipersNiederreiter1974}.
\end{definition}

For equidistributed  sequences and equidistributed double sequences, one have corresponding Weyl's criterion~\cite[Thm. 2.1 \& Thm. 2.9]{KuipersNiederreiter1974}.
\begin{theorem}~[Weyl's criteria]\

A (double) sequence $\{a_n\}_{n=1}^\infty$~($\{a_{i,j}\}_{i,j=0}^\infty$) is equidistributed on $\mathbb{T}$ iff
$$\lim_{N\to\infty}\dfrac{1}{N}\sum_{n=1}^N a_n^k=0$$
$$(\lim_{N,M\to\infty}\dfrac{1}{NM}\sum_{i=0}^{N-1}\sum_{j=0}^{M-1} a_{ij}^k=0),$$ for every $k\in\mathbb{Z}^+$.
\end{theorem}

Equivalently one have the following
\begin{theorem}~\cite[Thm. 1.1 \& Thm. 2.8]{KuipersNiederreiter1974}\

A (double) sequence $\{a_n\}_{n=1}^\infty$~($\{a_{i,j}\}_{i,j=0}^\infty$) in $\mathbb{T}$ is equidistributed iff
$$\lim_{N\to\infty}\dfrac{1}{N}\sum_{n=1}^N f(a_n)=\int_{\mathbb{T}}f(z)\,dm(z)$$
$$(\lim_{N,M\to\infty}\dfrac{1}{NM}\sum_{i=0}^{N-1}\sum_{j=0}^{M-1} f(a_{ij})=\int_{\mathbb{T}}f(z)\,dm(z)),$$ for every $f\in C(\mathbb{T})$. Here $m$ is the Lebesgue measure of $\mathbb{T}$.
\end{theorem}

A weaker version of equidistribution of double sequences is the following~\cite[The paragraph before Lemma 2.4]{KuipersNiederreiter1974}.
\begin{definition}\
A double sequence $\{a_{i,j}\}_{i,j=0}^\infty\subset\mathbb{T}$ is called {\bf equidistributed in the squares} on $\mathbb{T}$ if the sequence $\omega=\{x_{ij}\}_{i,j=0}^\infty$ in $[0,1)$ with $e(x_{ij})=a_{ij}$ satisfies
$$\lim_{N\to\infty}A([a,b); N,N;\omega)=b-a,$$ for any $0\leq a<b\leq 1$.
\end{definition}

Similarly $\{a_{i,j}\}_{i,j=0}^\infty\subset\mathbb{T}$ is  equidistributed in the squares on $\mathbb{T}$ iff
$$\lim_{N\to\infty}\dfrac{1}{N^2}\sum_{i=0}^{N-1}\sum_{j=0}^{N-1} a_{ij}^k=0$$ for every $k\in\mathbb{Z}^+$.

\section{Equidistributed double sequences and ergodic $\times p,\times q$ invariant measures}\

\subsection{Equidistributed irrational orbits}\
From now on, we fix two positive integers $p,q$ such that $\frac{\log{p}}{\log{q}}\notin\mathbb{Q}$~(the multiplicative semigroup $\{p^iq^j\}_{i,j\in\mathbb{N}}\nsubseteqq \{a^n\}_{n\in\mathbb{N}}$ for every $a\in\mathbb{Z}^+$.

In this section, we show that every ergodic $\times p,\times q$-invariant measure $\mu$ on $\mathbb{T}$ can be written as the weak-$*$ limit of
$\{\frac{1}{N^2}\sum_{i=0}^{N-1}\sum_{j=0}^{N-1}\mu_{a^{p^iq^j}}\}_{N=1}^\infty$ for some $a\in \mathbb{T}$. 

\begin{definition}~[Generic point]\

A point $a\in\mathbb{T}$ is called {\bf generic} with respect to an ergodic $\times p,\times q$-invariant measure $\mu$ of $\mathbb{T}$ if
$$\lim_{N\to\infty}\dfrac{1}{N^2}\sum_{i=0}^{N-1}\sum_{j=0}^{N-1}f(a^{p^iq^j})=\mu(f)$$ for all $f\in C(\mathbb{T})$. Denote the set of generic points with respect to $\mu$ by $X_\mu$.
\end{definition}

\begin{definition}~[Amenable semigroup]~\cite[p. 2]{Bowley1971}~\cite[p. 2]{OrnsteinWeiss1983}\

A countable discrete semigroup $P$ is called {\bf amenable} if there exists a sequence  $\{F_n\}_{n=1}^\infty$ of finite subsets of $P$ such that
$$\lim_{n\to\infty}\dfrac{|sF_n\bigtriangleup F_n|}{|F_n|}=0$$ for any $s\in P$, and  $\{F_n\}_{n=1}^\infty$ is called a (left) F{\o}lner sequence. A F{\o}lner sequence $\{F_n\}_{n=1}^\infty$ is called {\bf special} if

\begin{enumerate}
\item $F_n\subseteq F_{n+1}$;
\item There exists some constant $M>0$ such that $|F_n^{-1}F_n|\leq M|F_n|$ for all $n\in\mathbb{Z}^+$, where $F_n^{-1}F_n=\{s\in P|\,ts\in F_n \,\rm{for \,some} \, t\in F_n\}$.
\end{enumerate}

\end{definition}

Before proceeding to prove the main result, we need a pointwise ergodic theorem as a preliminary, which is a special case of~\cite[Thm. 3]{Bowley1971}.

\begin{theorem}~[Generalized Birkhoff pointwise ergodic theorem]~\label{PointwiseErgodic}\

Suppose $P$ is a discrete amenable semigroup and $X$ is a compact Hausdorff space. Assume that there is a continuous, measure-preserving action of $P$ on a Borel probability space $(X,\mathcal{B},\mu)$, and $\mu$ is an ergodic $P$-invariant measure. If $P$ has a special F{\o}lner sequence $\{F_n\}_{n=1}^\infty$, then for every $f\in L^1(X,\mu)$, the sequence
$\{\dfrac{1}{|F_n|}\sum_{s\in F_n} f(s\cdot x)\}_{n=1}^\infty$ converges almost everywhere to a $P$-invariant function  $f^*\in L^1(X,\mu)$ such that $\int_X f\,d\mu=\int_X f^*\,d\mu$.
\end{theorem}

Using Theorem~\ref{PointwiseErgodic}, we prove the following theorem which shows generic points with respect to an ergodic $\times p,\times q$-invariant measure $\mu$ are almost everywhere.

\begin{theorem}~\label{GenericEverywhere}
For every ergodic $\times p,\times q$-invariant measure $\mu$ on $\mathbb{T}$, we have $\mu(X_\mu)=1$.
\end{theorem}

\begin{proof}
Consider the measure preserving action of $\mathbb{N}^2$ on $(\mathbb{T},\mu)$ given by $\times p,\times q$. Note that $\mathbb{N}^2$ is an amenable semigroup with a special F{\o}lner sequence $\{F_{n}\}_{n=1}^\infty$  given by $F_{n}=\{(i,j)|0\leq i,j\leq n-1\}$.   Since $\mu$ is ergodic, every $\mathbb{N}^2$-invariant function in $L^1(\mathbb{T},\mu)$ is constant. Applying Theorem~\ref{PointwiseErgodic},  we have
$$\lim_{N\to\infty}\dfrac{1}{N^2}\sum_{i=0}^{N-1}\sum_{j=0}^{N-1}S^iT^j(f)(x)=\mu(f)$$ for every $f\in C(\mathbb{T})$ and almost every $x\in\mathbb{T}$ with respect to $\mu$.  Denote the set of such points for $f$ by $X_f$. Then $\mu(X_f)=1$.

Take a countable dense set $\{f_n\}_{n=1}^\infty$ in $C(\mathbb{T})$. Then it is easy to see that $X_\mu=\bigcap_{n=1}^\infty X_{f_n}$ and hence $\mu(X_\mu)=1$.
\end{proof}

\begin{corollary}~\label{Support}
If $\mu$ is finitely supported, then $Supp(\mu)$, the support of $\mu$ is a subset of $X_\mu$.
\end{corollary}
\begin{proof}
Since $\mu$ is atomic, the set $Supp(\mu)$ consists of finitely many atoms. Hence every atom is in $X_\mu$ otherwise $\mu(X_\mu)<1$.
\end{proof}

Next we prove that every rational is a generic point with respect to an atomic ergodic $\times p,\times q$-invariant measure.
\begin{lemma}~\label{Orbit}\
If $x,y\in [0,1)$ are in the same orbit under $\times p,\times q$~(which means $x=p^iq^iy \mod 1$ for some $i,j\in\mathbb{Z}$, then $x\in X_\mu$ iff $y\in X_\mu$.
\end{lemma}
\begin{proof}
Let $a=e(x)$ and $b=e(y)$. There exists $c\in\mathbb{T}$ such that $c=a^{p^mq^n}=b^{p^kq^l}$ for some $k,l,m,n\in\mathbb{N}$. The proof follows from
\begin{align*}
\lim_{N\to\infty}\dfrac{1}{N^2}\sum_{i=0}^{N-1}\sum_{j=0}^{N-1}f(a^{p^iq^j})&=\lim_{N\to\infty}\dfrac{1}{N^2}\sum_{i=0}^{N-1}\sum_{j=0}^{N-1}f(c^{p^iq^j})  \\
&=\lim_{N\to\infty}\dfrac{1}{N^2}\sum_{i=0}^{N-1}\sum_{j=0}^{N-1}f(b^{p^iq^j})
\end{align*}
(if any of these three limits exists) for all $f\in C(\mathbb{T})$.
\end{proof}

A finite Borel measure $\mu$ on $\mathbb{T}$ is called {\bf continuous} or {\bf non-atomic} if $\mu\{z\}=0$ for every $z\in\mathbb{T}$.

\begin{theorem}~\label{IrrationalOrbit}
Every rational $a\in\mathbb{T}$ is a generic point with respect to a finitely supported ergodic $\times p,\times q$-invariant measure. Hence for an ergodic $\times p,\times q$-invariant continuous  measure $\mu$ on $\mathbb{T}$, there exists an irrational $x\in[0,1)$ such that $$\lim_{N\to\infty}\dfrac{1}{N^2}\sum_{i=0}^{N-1}\sum_{j=0}^{N-1}\delta_{p^iq^jx}=\mu$$ under weak-$*$ topology.
\end{theorem}
\begin{proof}
Let $a=e(\frac{m}{n})$ for  $m,n\in\mathbb{Z}^+$ with $\gcd(m,n)=1$. Then there exist $s,t\in\mathbb{Z}^+$ such that
\begin{itemize}
\item $\frac{m}{n}$ and $\frac{s}{t}$ are in the same orbit under $\times p,\times q$.
\item $\gcd(s,t)=1$ and $\gcd(t,pq)=1$.
\end{itemize}

There exists an ergodic $\times p,\times q$-invariant measure $\mu$ finitely supported in $\{\frac{j}{s}|0\leq j\leq s-1\}$ such that $\frac{s}{t}$ is in $Supp(\mu)$. Combining Corollary~\ref{Support} and Lemma~\ref{Orbit}, we finish the proof of the first part. The second part follows immediately.
\end{proof}

\section{Invariant subspace of $C[0,1]$ under an action of a multiplicative semigroup $\Sigma$ of $\mathbb{N}$.}\

\subsection{$\times p,\times q$-invariant measures via continuous functions on $[0,1]$}\
\begin{definition}~\label{T_n}
For a positive integer $n\geq 2$, define $T_n:C[0,1]\to C[0,1]$ by
$$T_n(g)(x)=\sum_{j=0}^{n-1} [g(\frac{x+j}{n})-g(\frac{j}{n})]$$ for all $g\in C[0,1]$. We say an $f\in C[0,1]$ is {\bf $T_n$-invariant} if $T_n(f)=f$.
\end{definition}
 We see that $T_n$ is a bounded linear operator under the  norm $\|g\|:=\max_{x\in [0,1]}{|g(x)|}$ for all $g\in C[0,1]$.

\begin{lemma}
For $n\geq 2$, if $f$ is $T_n$-invariant, then
\begin{equation}~\label{eqf}
\int_0^1 f(x)\,dx=\frac{\sum_{j=0}^{n-1}f(\frac{j}{n})}{n-1}.
\end{equation}
\end{lemma}
\begin{proof}

Take integral from 0 to 1 on both sides, we get
\begin{align*}
\int_0^1 f(x)\,dx&=\int_0^1 [\sum_{j=0}^{n-1} [f(\frac{x+j}{n})-f(\frac{j}{n})]\,dx  \\
&=\sum_{j=0}^{n-1}n\int_{\frac{j}{n}}^{\frac{j+1}{n}} f(x)\,dx-\sum_{j=0}^{n-1}f(\frac{j}{n}) \\
&=n\int_0^1 f(x)-\sum_{j=0}^{n-1}f(\frac{j}{n}).
\end{align*}
Then Equation~\ref{eqf} follows immediately.
\end{proof}

\begin{proposition}~\label{asso}\
For two positive integers $n$ and $m$, we have $T_nT_m=T_{nm}$.
\end{proposition}
\begin{proof}
\begin{align*}
&T_nT_mf(x)=\sum_{j=0}^{n-1}[T_mf(\frac{x+j}{n})-T_mf(\frac{j}{n})]   \\
&=\sum_{j=0}^{n-1}\sum_{k=0}^{m-1}\{[f(\frac{\frac{x+j}{n}+k}{m})-f(\frac{k}{m})]-[f(\frac{\frac{j}{n}+k}{m})-f(\frac{k}{m})]\} \\
&=\sum_{j=0}^{n-1}\sum_{k=0}^{m-1}[f(\frac{x+j+nk}{mn})-f(\frac{j+nk}{mn})]\\
&=\sum_{j=0}^{mn-1}[f(\frac{x+j}{mn})-f(\frac{j}{mn})] =T_{nm}f(x).
\end{align*}
\end{proof}

Let $\Sigma$ be a multiplicative semigroup of $\mathbb{N}$. By Proposition~\ref{asso}, we have a semigroup action of $\Sigma$ on $C[0,1]$ given by $T_n$ for all $n\in\Sigma$. We say an $f\in C[0,1]$ is $\Sigma$-invariant if $f$ is $T_n$-invariant for all $n\in\Sigma$.

Next we show some connection between continuous $\Sigma$-invariant measures on $\mathbb{T}$ and $\Sigma$-invariant functions in $C[0,1]$.

Given a probability measure $\mu$ on $\mathbb{T}$, identify $\mathbb{T}$ with $[0,1)$. Then $\mu$ can be taken as a probability measure on $[0,1)$, and  if $\mu$ is continuous, then $D_\mu$ is a  nondecreasing continuous function on $[0,1]$ and $D_\mu(0)=0, D_\mu(1)=1$.

\begin{definition}~\label{dist}\
For a probability measure $\mu$ on $[0,1)$, define the {\bf distribution function} of $\mu$, denoted by $D_\mu$, by $D_\mu(x)=\mu[0,x)$ for all $x\in[0,1]$.
\end{definition}

\begin{proposition}~\label{invmf}\
Suppose $\mu$ is a continuous measure on $\mathbb{T}$. Then $\mu$ is $\times n$-invariant iff $D_\mu$ is $T_n$-invariant.
\end{proposition}
\begin{proof}
Suppose $\mu$ is $\times n$-invariant. For any $x\in[0,1]$, the preimage of $[0,x)$ under $\times n$ is $\bigcup_{i=0}^{n-1}[\frac{i}{n},\frac{x+i}{n})$. So we get
$$D_\mu(x)=\mu[0,x)=\sum_{i=0}^{n-1}\mu[\frac{i}{n},\frac{x+i}{n})=\sum_{i=0}^{n-1}D_\mu(\frac{x+i}{n})-D_\mu(\frac{i}{n})=T_nD_\mu(x).$$

On the other hand, assume that $D_\mu$ is $T_n$-invariant. To show that $\mu$ is $\times n$-invariant, we only need to check that $\mu(z^k)=\mu(z^{kn})$ for all positive integers $k$. Here $z=e^{2\pi i x}$.
\begin{align*}
\mu(z^k)&=\int_0^1 e^{2\pi ikx}\,dD_\mu(x)= e^{2\pi ikx}D_\mu(x)|_0^1-\int_0^1 D_\mu(x)\,de^{2\pi ikx}  \\
&=1-2\pi ik\int_0^1 [\sum_{j=0}^{n-1}D_\mu(\frac{x+j}{n})-D_\mu(\frac{j}{n})]\,e^{2\pi ikx}dx  \\
&=1-2\pi ik\int_0^1 [\sum_{j=0}^{n-1}D_\mu(\frac{x+j}{n})]\,e^{2\pi ikx}dx   \\
&=1-2\pi ikn\sum_{j=0}^{n-1}\int_{\frac{j}{n}}^{\frac{j+1}{n}}D_\mu(x)e^{2\pi ik(nx-j)}\,dx \\
&=1-2\pi ikn\int_0^1D_\mu(x)e^{2\pi iknx}\,dx=\mu(z^{kn}).
\end{align*}
\end{proof}

For a semigroup $\Sigma\subseteq\mathbb{N}$, denote the space of $\Sigma$-invariant functions by $C_\Sigma[0,1]$.
\begin{theorem}\
The Lebesgue measure is the only continuous $\Sigma$-invariant measure on $\mathbb{T}$ if $\dim{C_\Sigma[0,1]}=1$.
\end{theorem}
\begin{proof}
Note that $T_nx=x$ for all positive integer $n$. Hence  $\dim{C_\Sigma[0,1]}\geq1$. If $\dim{C_\Sigma[0,1]}=1$, then $C_\Sigma[0,1]$ consists of functions of the form $\alpha x$ for some complex number $\alpha$. Suppose $\mu$ is a continuous $\Sigma$-invariant measure on $\mathbb{T}$. By Proposition~\ref{invmf}, the distribution function $D_\mu$ is in $C_\Sigma[0,1]$ and $D_\mu(1)=1$. So $D_\mu(x)=x$, which means that $\mu$ is the Lebesgue measure.
\end{proof}

Consequently, if the following conjecture is true, then Furstenberg's conjecture is true.
\begin{conjecture}~\label{Cont}\
The only $f\in C[0,1]$ satisfying that

\begin{enumerate}
\item $f$ is non-decreasing~(even by Furstenberg's classification result of closed $\times p,\times q$-invariant subsets of $\mathbb{T}$, we can assume that $f$ is strictly increasing, hence a homeomorphism on $[0,1]$ with $f(0)=0$);
\item $f(x)=\sum_{i=0}^{p-1} f(\frac{x+i}{p})-f(\frac{i}{p})=\sum_{i=0}^{q-1} f(\frac{x+i}{q})-f(\frac{i}{q})$,
\end{enumerate}
is $x$.
\end{conjecture}

\subsection{The Cantor function as a $T_3$-invariant function}\
Although Furstenberg's conjecture is equivalent to a conjecture in the framework of calculus, the difficulty doesn't reduce at all. To get a feeling of  this, we look at a concrete example, the Cantor function, which is $T_3$-invariant, but not a homeomorphism on $[0,1]$.

\begin{definition}~[The Cantor function]~\label{Cantor}\

The {\bf Cantor function} $c:[0,1]\to[0,1]$ is defined via the following procedures:
\begin{enumerate}
\item Express $x\in[0,1]$ in base 3;
\item If $x$ contains a 1, replace every digit after the first 1 by 0;
\item Replace all 2s with 1s;
\item Interpret the result as a binary number.
\end{enumerate}
The result is $c(x)$.
\end{definition}

Notice that $c(\frac{1}{3})=c(\frac{2}{3})=\frac{1}{2}$, so $c(x)$ is not a homeomorphism although it is a non-decreasing map from $[0,1]$ onto $[0,1]$.

Let $m(x)=\min\{n|x_n=1\}$ for $x=\sum_{n=1}\frac{x_n}{3^n}\in[0,1]$ with $0\leq x_n \leq 2$. When there is no $n$ such that $x_n=1$, let $m(x)=\infty$.
Hence
\begin{equation}~\label{eqc}
c(x)=\sum_{n<m(x)}\dfrac{\frac{x_n}{2}}{2^n}+\frac{1}{2^{m(x)}}
\end{equation}
 for all $x\in [0,1]$. If $m(x)=1$, then $c(x)=\frac{1}{2}$. Of course, Equation~\ref{eqc} is nothing new, but in some sense, it is more explicit~(hence more convenient) for us to prove some properties of $c(x)$.

Using Equation~\ref{eqc}, the proof of the following lemma is straightforward.

\begin{lemma}~[Properties of $c(x)$]~\label{pc}\

For every $x\in[0,1]$, we have
\begin{enumerate}
\item $c(\frac{x}{3})=\frac{1}{2}c(x)$;
\item $c(\frac{x+1}{3})=\frac{1}{2}$;
\item $c(\frac{x+2}{3})=c(\frac{2}{3})+c(\frac{x}{3})=\frac{1}{2}+\frac{1}{2}c(x)$.
\end{enumerate}
\end{lemma}

Lemma~\ref{pc} gives the following.

\begin{proposition}~\label{InvC}
The Cantor function $c(x)$ is $T_3$-invariant.
\end{proposition}
\begin{proof}
By definition $T_3 c(x)=[c(\frac{x}{3})+c(\frac{x+1}{3})+c(\frac{x+2}{3})]-[c(0)+c(\frac{1}{3})+c(\frac{2}{3})]$ for all $x\in[0,1]$. It follows from Lemma~\ref{pc} that
$$T_3 c(x)=[\frac{c(x)}{2}+\frac{1}{2}+\frac{1}{2}+\frac{c(x)}{2}]-[0+\frac{1}{2}+\frac{1}{2}]=c(x)$$ for all $x\in[0,1]$.
\end{proof}

\begin{lemma}~\label{nL}
$c(x)$ is not a Lipschitz function.
\end{lemma}
\begin{proof}
Take $x=1$ and $y$ such that $y_n=2$ for $n<N$, $y_N=1$ and $y_n=0$ for $n>N$. Then $|x-y|=\frac{2}{3^N}$ and $c(x)-c(y)=1-(\frac{1}{2}+\cdots+\frac{1}{2^N})=\frac{1}{2^N}$.
Hence $$\frac{|c(x)-c(y)|}{|x-y|}=\frac{1}{2}(\frac{3}{2})^N.$$ As $N\to\infty$, we have $y\to x$, but $\frac{|c(x)-c(y)|}{|x-y|}\to\infty$.
\end{proof}

By Equation~\ref{eqf}, we have
$$\int_0^1 c(x)\, dx=\frac{c(\frac{1}{3})+c(\frac{2}{3})}{2}=\frac{1}{2}.$$


\begin{thebibliography}{999}

\small

\bibitem[AB00]{AnoussisBisbas2000}
M. Anoussis and A. Bisbas. Continuous measures on compact Lie groups. {\it Ann. Inst. Fourier (Grenoble)} {\bf 50} (2000), no. 4, 1277-–1296.

\bibitem[BM15]{BergelsonMoreira2015}	
V. Bergelson and J. Moreira. Van der Corput's Difference Theorem: some modern developments. arXiv:1510.07332v1.

\bibitem[B71]{Bowley1971}
T. Bowley. Extension of the Birkhoff and von Neumann ergodic theorems to semigroup actions.
{\it Ann. Inst. H. Poincar\'{e}}  Sect. B (N.S.) {\bf 7} (1971), 283–-291.



\bibitem[C66]{Carleson1966}
L. Carleson. On convergence and growth of partial sums of Fourier series. {\it Acta Math}. {\bf 116} (1966), 135-–157.

\bibitem[vC31]{vanderCorput1931}
van der J. G. Corput. Diophantische Ungleichungen. I. Zur Gleichverteilung Modulo Eins.
{\it Acta Math.}, {\bf 56} (1931), 1, 373-–456.

\bibitem[D11]{Deninger2011}

C. Deninger. Invariant measures on the circle and functional equations. arXiv:1111.6416.

\bibitem[ELT]{ELT}
M. Einsiedler, E. Lindenstrauss and T. Ward. {\it Entropy in Ergodic Theory and Homogeneous Dynamics.} A book in preparation.

\bibitem[F76]{Furstenberg1967}
H. Furstenberg. Disjointness in ergodic theory, minimal sets, and a problem in Diophantine approximation. {\it Math. Systems Theory.} {\bf 1} (1967) 1–-49.

\bibitem[G14]{Grafakos2014}
L. Grafakos. {\it Classical Fourier Analysis}. Third edition. Graduate Texts in Mathematics, {\bf 249}. Springer, New York, 2014.

\bibitem[H95]{Host1995}
B. Host. Nombres normaux, entropie, translations. {\it Israel J. Math.} {\bf 91} (1995), no. 1-3, 419-–428.

\bibitem[H15]{Huang2015}
H. Huang. Mean ergodic theorem for coamenable compact quantum groups and a Wiener type theorem for compact metrizable groups. 2015.

\bibitem[HW15]{HW2015}
H. Huang and J. Wu. Ergodic invairant states and irreducible representations of crossed product $C^*$-alegbras. 2015.

\bibitem[H68]{Hunt1968}
R. A. Hunt. On the convergence of Fourier series. 1968 {\it Orthogonal Expansions and their Continuous Analogues (Proc. Conf., Edwardsville, Ill.}, 1967) pp. 235-–255 Southern Illinois Univ. Press, Carbondale, Ill.


\bibitem[K04]{Katznelson2004}
Y. Katznelson. {\it An Introduction to Harmonic Analysis.} Third edition. Cambridge Mathematical Library. Cambridge University Press, Cambridge, 2004.

\bibitem[KN74]{KuipersNiederreiter1974}
L. Kuipers and H. Niederreiter. {\it Uniform Distribution of Sequences.} Pure and Applied Mathematics. Wiley-Interscience, 1974.

\bibitem[OW83]{OrnsteinWeiss1983}
D. Ornstein and B. Weiss. The Shannon-McMillan-Breiman theorem for a class of amenable groups. {\it Israel J. Math.} {\bf 44} (1983), no. 1, 53–-60.

\bibitem[R00]{Robert2000}
A. Robert. {\it A Course in p-adic Analysis.} Graduate Texts in Mathematics, {\bf 198}. Springer-Verlag, New York, 2000.

\bibitem[R90]{Rudolph1990}
D. J. Rudolph. $\times 2$ and $\times 3$ invariant measures and entropy. {\it Ergod. Th. and Dynam. Syst.} {\bf 10}, (1990), 395--406.



\end{thebibliography}
\end{document}